\documentclass{amsart}
\usepackage{mathrsfs,amssymb,xy,natbib}

\theoremstyle{plain}
\newtheorem{theorem}{Theorem}[section]
\newtheorem{lemma}[theorem]{Lemma}
\newtheorem{proposition}[theorem]{Proposition}
\newtheorem{corollary}[theorem]{Corollary}

\theoremstyle{definition}
\newtheorem{notation}[theorem]{Notation}
\newtheorem{example}[theorem]{Example}
\newtheorem{definition}[theorem]{Definition}

\theoremstyle{remark}
\newtheorem{remark}[theorem]{Remark}

\def\g{\gamma}
\def\G{\Gamma}
\def\vt{\vartheta}
\def\vn{\varnothing}
\newcommand{\sis}[2]{^{#1}\!\zeta_{#2}}
\begin{document}


\title[Affine Near-Semirings over Brandt Semigroups]{Affine Near-Semirings over Brandt Semigroups}

\author[Jitender Kumar, K. V. Krishna]{Jitender Kumar and  K. V. Krishna}
\address{Department of Mathematics, Indian Institute of Technology Guwahati, Guwahati, India}
\email{\{jitender, kvk\}@iitg.ac.in}


\begin{abstract}
In order to study the structure of $A^+(B_n)$ -- the affine near-semiring over a Brandt semigroup -- this work completely characterizes the Green's classes of its semigroup reducts. In this connection, this work classifies the elements of $A^+(B_n)$ and reports the size of $A^+(B_n)$. Further, idempotents and regular elements of the semigroup reducts of $A^+(B_n)$ have also been characterized and studied some relevant semigroups in $A^+(B_n)$.
\end{abstract}

\subjclass[]{16Y30, 16Y60, 20M10}

\keywords{Semigroup structure, Near-semiring, Green's relations, Brandt semigroup, Affine maps}

\maketitle


\section*{Introduction}

An algebraic structure $(S, +, \cdot)$ with two binary operations $+$ and $\cdot$ is said to be a near-semiring if $(S, +)$ and $(S, \cdot)$ are semigroups and $\cdot$ is one-side, say left, distributive over $+$, i.e. $a(b + c) = ab + ac$, for all $a,b,c \in S$. Typical examples of near-semirings are of the form $M(\Gamma)$, the set of all mappings on a semigroup $\Gamma$, and certain subsets of $M(\G)$. If $\G$ is a group, then $M(\G)$ is endowed with the structure called near-ring \citep{b.pilz}.

\cite{a.hoorn67} have introduced the concept of near-semirings as a generalization of near-rings and established some fundamental properties. Several authors have studied near-semirings in various aspects. Some authors have considered studying the algebraic structure of near-semirings (e.g. \cite{a.albert76,a.kvk07a,a.hoorn70,a.weinert82}) and others utilized the concept in various applications (e.g. \cite{jules08}). Recently, \cite{a.gilbert10a,a.gilbert10b} have studied the classes of endomorphism near-semirings over Clifford semigroups and Brandt semigroups.

An affine mapping over a vector space is a sum of a linear transformation and a constant map. \cite{bl56} studied the near-ring of affine mappings over a vector space. An abstract notion of affine near-rings is introduced by  \cite{gh64}. Authors have considered the study of affine near-rings in different contexts (e.g. \cite{shalom85,mal69}). \cite{a.hol83,a.hol84} studied affine near-rings, in the context of linear sequential machines. These notions are extended to near-semirings by \cite{kvk05a}. Further, \cite{kvk05b}  have studied affine near-semirings over generalized linear sequential machines. A mapping on a semigroup $(\Gamma, +)$ is said to be an affine map if it can be written as a sum of an endomorphism and a constant map. The subnear-semiring of $M(\G)$ generated by the set of affine maps is an affine near-semiring over $\G$, denoted by $A^+(\G)$.

In this work, we study $A^+(B_n)$ -- the affine near-semiring over a Brandt semigroup $B_n$. In this connection, we report the size of $A^+(B_n)$ and study the structural properties of semigroup reducts of $A^+(B_n)$ via Green's relations. Other than the introduction, the paper has been organized in five sections. Section 1 provides a necessary background material. In Section 2, we conduct a systematic study to classify the elements of $A^+(B_n)$ and finally report its size. Sections 3 and 4 are devoted to give the structures of additive and multiplicative semigroups of $A^+(B_n)$, respectively. In the respective sections, for both the semigroups we provide complete characterizations of Green's classes and their sizes. Further, we characterize regular and idempotent elements of both the semigroups and study some relevant subsemigroups. Finally, in Section 5, we illustrate our results on the affine near-semiring $A^+(B_2)$.

\section{Preliminaries}

In this section, we provide a necessary background material and fix our notation. For more details one may refer to \citep{a.gilbert10b,kvk05a}.

\begin{definition}
An algebraic structure $(S, +, \cdot)$ is said to be a \emph{near-semiring} if
\begin{enumerate}
\item $(S, +)$ is a semigroup,
\item $(S, \cdot)$ is a semigroup, and
\item $a(b + c) = ab + ac$, for all $a,b,c \in S$.
\end{enumerate}
\end{definition}

In this work, unless it is required, algebraic structures (such as semigroups, groups, near-semirings) will simply be referred by their underlying sets without explicit mention of their operations. Further, we write an argument of a function on its left, e.g. $xf$ is the value of a function $f$ at an argument $x$.

\begin{example}
Let $(\Gamma, +)$ be a semigroup and $M(\Gamma)$ be the set of all mappings on $\Gamma$. The algebraic structure $(M(\G), +, \circ)$ is a near-semiring, where $+$ is point-wise addition and $\circ$ is composition of mappings, i.e., for $\gamma \in \Gamma$ and $f,g \in M(\Gamma)$,
$$\g(f + g)= \g f + \g g \;\;\;\; \text{and}\;\;\;\; \g(f \circ g) = (\g f)g.$$ Also, certain subsets of $M(\G)$ are near-semirings. For instance, the set $M_c(\Gamma)$ of all constant mappings on $\Gamma$ is a near-semiring with respect to the above operations so that $M_c(\Gamma)$ is a subnear-semiring of $M(\G)$.
\end{example}

Given a semigroup $(\G, +)$, the set $End(\G)$ of all endomorphisms over $\G$ need not be a subnear-semiring of $M(\G)$.  The \emph{endomorphism near-semiring}, denoted by $E^+(\G)$, is the subnear-semiring generated by $End(\G)$ in $M(\G)$. Indeed, the subsemigroup of $(M(\G), +)$ generated by $End(\G)$ equals $(E^+(\G), +)$. If $(\G, +)$ is commutative, then $End(\G)$ is a subnear-semiring of $M(\G)$ so that $End(\G) = E^+(\G)$. \cite{a.gilbert10b}  have studied endomorphism near-semirings over Brandt semigroups.

\begin{definition}
For any integer $n \geq 1$, let $[n] = \{1,2,\ldots,n\}$. The semigroup $(B_n, +)$, where $B_n = ([n]\times[n])\cup \{\vartheta\}$ and the operation $+$ is given by
\[ (i,j) + (k,l) =
                \left\{\begin{array}{cl}
                (i,l) & \text {if $j = k$;}  \\
                \vartheta     & \text {if $j \neq k $}
                  \end{array}\right.  \]
and, for all $\alpha \in B_n$, $\alpha + \vartheta = \vartheta + \alpha = \vartheta$,
is known as \emph{Brandt semigroup}. Note that $\vartheta$ is the (two sided) zero element in $B_n$.
\end{definition}

Gilbert and Samman have studied the structure of additive semigroup of near-semiring $E^+(B_n)$ via Green's relations. Further, they have reported the sizes of $E^+(B_n)$, for $n \leq 6$, using the computer algebra package GAP. The following concept plays a vital role in their work.

\begin{definition}
Let $(\Gamma, +)$ be a semigroup with zero element $\vartheta$. For  $f \in M(\Gamma)$, the \emph{support of $f$}, denoted by supp$(f)$, is defined by the set
\[ {\rm supp}(f) = \{\alpha \in \Gamma \;|\; \alpha f \neq \vartheta\}.\]
A function $f \in M(\Gamma)$ is said to be of $k$-\emph{support} if the cardinality of supp$(f)$ is $k$, i.e. $|{\rm supp}(f)| = k$. If $k = |\Gamma|$ or $k = 1$, then $f$ is said to be of \emph{full support} or \emph{singleton support}, respectively.
\end{definition}

Note that the notion of full support maps used in \citep{a.gilbert10b} is too restrictive and is limited to the elements of $E^+(B_n)$ whose support size is at most $n$. Whereas, the above definition can be adopted to any map.

\begin{notation}
For $X \subseteq M(\Gamma)$, we write $X_k$ to denote the set of all mappings of $k$-support in $X$, i.e.
$$ X_k = \{ f \in X \mid f \; \text{is of $k$-support} \}.$$
\end{notation}

\begin{lemma}\label{r2}
If $f,g \in M(\G)$, then ${\rm supp}(f+g) \subseteq {\rm supp}(f)\cap {\rm
supp}(g)$.  Moreover,  for  $f \in M(\Gamma)_k$, we have   $|{\rm
supp}(f+g)| \leq k$ and $|{\rm supp}(g + f)| \leq k$.
\end{lemma}

\begin{proof}
Straightforward.
\end{proof}

The present work on affine near-semirings is inspired by the work of Gilbert and Samman on endomorphism near-semirings.
Now, we recall the notion of affine maps and affine near-semirings from \citep{kvk05a}. Let $(\G, +)$ be a semigroup. An element $f \in M(\G)$ is said to be an \emph{affine map} if $f = g + h$, for some $g \in End(\G)$ and $h \in M_c(\G)$. The set $\text{Aff}(\G)$  of all affine mappings over $\G$ need not be a subnear-semiring of $M(\G)$. The \emph{affine near-semiring}, denoted by $A^+(\G)$, is the subnear-semiring generated by $\text{Aff}(\G)$ in $M(\G)$. Indeed, the subsemigroup of $(M(\G), +)$ generated by $\text{Aff}(\G)$ equals $(A^+(\G), +)$ (cf. \cite[Corollary 1]{kvk05b}). If $(\G, +)$ is commutative, then $\text{Aff}(\G)$ is a subnear-semiring of $M(\G)$ so that $\text{Aff}(\G) = A^+(\G)$.

For $\alpha \in \G$, the constant map on $\G$ which sends all the elements of $\G$ to $\alpha$ is denoted by $\xi_\alpha$, i.e. $\gamma\xi_\alpha = \alpha$ for all $\gamma \in \G$. For $X \subseteq \G$, we write \[\mathcal{C}_X = \{\xi_\alpha \in M(\G) \mid \alpha \in X\}.\] Note that $\mathcal{C}_{\G} = M_c(\G)$. If we are not specific about the constant image, we may simply write $\xi$ to denote a constant map. We write the set of idempotent elements of $\G$ by $I(\G)$. Note that $I(B_n) = \{(k, k) : k \in [n]\} \cup \{\vt \}$.

\section{The elements of $A^+(B_n)$}

In this section, we carry out a systematic study through two subsections to characterize the elements of $A^+(B_n)$ and find the size of $A^+(B_n)$. First we study the elements of $End(B_n)$ and ${\rm Aff}(B_n)$ in Subsection 2.1. Then, in Subsection 2.2, we obtain a main result on classification and the number of elements of $A^+(B_n)$ (cf. Theorem \ref{t2}).

\subsection{$End(B_n)$ and ${\rm Aff}(B_n)$}\label{si-eb}

We find the size of $End(B_n)$ by generalizing the corresponding result for $End_\vartheta(B_n)$ in  \citep{a.gilbert10b}. They have shown that the monoid \[End_\vartheta(B_n) = \{f \in End(B_n)\; |\; \vt f = \vt \},\] with respect to composition of mappings, is isomorphic to the monoid $S_n^{0}  = S_n \cup \{0\}$, where the symmetric group $S_n$ of degree $n$ is adjoined by the zero element $0$ (cf. \cite[Proposition 2.2]{a.gilbert10b}). Thus, it is clear that $|End_\vartheta(B_n)| = n! + 1$. Now, we extend this result to $End(B_n)$ and find its cardinality in Theorem \ref{t1}.

Let $f \in Aut(B_n)$, the set of all automorphisms over $B_n$. Clearly $\vartheta f = \vartheta$. Then, from the proof of \cite[Proposition 2.2]{a.gilbert10b}, there exists a permutation $\sigma \in S_n$ such that $(i,j)f = (i\sigma, j\sigma)$. Further, for any permutation  $\sigma \in S_n$, the mapping $\phi_\sigma : B_n \rightarrow B_n$ such that $(i,j) \mapsto (i\sigma,j\sigma)$, $\vt \mapsto \vt$ is an automorphism over $B_n$. Now, it can be observed that the assignment
\[\sigma \mapsto \phi_\sigma: S_n \rightarrow  Aut(B_n)\]
is an isomorphism. Thus, we have the following proposition.

\begin{proposition}\label{bn-iso-sn}
$Aut(B_n)$ is isomorphic to $S_n$.
\end{proposition}

\begin{theorem}\label{t1}
$End(B_n) = Aut(B_n) \cup \mathcal{C}_{I(B_n)}$. Hence,  $|End(B_n)| = n!+n+1.$
\end{theorem}

\begin{proof}
Clearly,  $Aut(B_n) \cup \mathcal{C}_{I(B_n)} \subseteq End(B_n)$. Let $f \in End(B_n)$. If $\vt f = \vt$, then \break by \cite[Proposition 2.2]{a.gilbert10b}, $f \in Aut(B_n)$ and $f = \phi_{\sigma}$ for some $\sigma \in S_n$. If  $\vt f \ne \vt$, then $\vt f = (k, k)$ for some $k \in [n]$. Now, for any $(i, j)\in B_n$, $$(k, k) = \vt f = ((i, j) + \vt)f = (i, j)f + (k, k)$$ so that $(i, j)f = (k, k)$. Thus, if $\vt f \ne \vt$, then $f \in \mathcal{C}_{I(B_n)}$. Hence, \[End(B_n) = Aut(B_n) \cup \mathcal{C}_{I(B_n)}.\]

Since $|\mathcal{C}_{I(B_n)}| = n + 1$ and, by Proposition \ref{bn-iso-sn}, $|Aut(B_n)| = |S_n| = n!$, we have $|End(B_n)| = n!+n+1$.
\end{proof}

We now characterize and count the elements of ${\rm Aff}(B_n)$ in the following theorem.

\begin{theorem}\label{t3}
${\rm Aff}(B_n) =  {\rm Aff}(B_n)_n \cup \mathcal{C}_{B_n}$. Moreover, $|{\rm Aff}(B_n)| = (n!+1)n^2+1.$
\end{theorem}

\begin{proof}
Since every constant map over $B_n$ is an affine map, we have \break ${\rm Aff}(B_n)_n \cup \mathcal{C}_{B_n} \subseteq {\rm Aff}(B_n)$. Let $f \in {\rm Aff}(B_n)$. If $f = \xi_{\vt}$, then clearly $f \in \mathcal{C}_{B_n}$. Otherwise, write $f = g + \xi_{(p,q)}$ for some $g \in End(B_n) \setminus \{\xi_{\vt}\}$. By Theorem \ref{t1}, $g$ can be either $\xi_{(k, k)}$ for some $k \in [n]$ or $\phi_{\sigma}$ for some $\sigma \in S_n$.

In case $g = \xi_{(k, k)}$ for some $k \in [n]$, since $f \ne \xi_{\vt}$, we have $k = p$ so that $f = \xi_{(p, q)} \in \mathcal{C}_{B_n}$. Further, as there are $n$ possibilities each for $p$ and $q$, we have $n^2$ affine maps (of full support) in this case.

We may now suppose $g = \phi_{\sigma}$ for some $\sigma \in S_n$. Clearly, we have $\vt f = \vt$, because $\vt \phi_{\sigma} = \vt$. Now for $(i, j) \in B_n \setminus \{\vt\}$
\[(i,j)f = (i \sigma, j \sigma) + (p, q) =  \left\{\begin{array}{cl}
                (i\sigma, q), & \text {if $j = p \sigma^{-1}$};  \\
                \vt,     & \text {otherwise.}
                  \end{array}\right.  \]
Hence, ${\rm supp}(f) = \{(i, p \sigma^{-1}) : i \in [n]\}$ so that $f \in {\rm Aff}(B_n)_n$. Consequently,
\[{\rm Aff}(B_n) =  {\rm Aff}(B_n)_n \cup \mathcal{C}_{B_n}.\]

Since the above union is disjoint and $|\mathcal{C}_{B_n}| = n^2 + 1$, it remains to prove that $|{\rm Aff}(B_n)_n| = (n!)n^2$. As shown above, every affine map of $n$-support is precisely of the form $\phi_\sigma + \xi_{(p, q)}$, for some $\sigma \in S_n$ and $p, q \in [n]$.  Thus, $|{\rm Aff}(B_n)_{n}| \le (n!)n^2$. Now, let $f = \phi_\sigma + \xi_{(p, q)}$ and $g = \phi_\rho + \xi_{(s, t)}$, for some $\sigma, \rho \in S_n$ and $p, q, s, t \in [n]$. If $q \ne t$, then clearly ${\rm Im}(f) \ne {\rm Im}(g)$. If $\sigma \ne \rho$, then there exists $i_0 \in [n]$ such that $i_0 \sigma \neq i_0 \rho$. If ${\rm supp}(f) \cap {\rm supp}(g) = \vn$, then $f \ne g$. Otherwise, for $(i_0, k) \in {\rm supp}(f) \cap {\rm supp}(g)$, \[(i_0, k)f = (i_0\sigma, q) \ne (i_0 \rho, t) = (i_0, k)g.\] Assume $\sigma = \rho$ but $p \ne s$, then clearly ${\rm supp}(f) \ne {\rm supp}(g)$. Thus, distinct choices of $\sigma$ and $(p, q)$ determine distinct affine maps of $n$-support. Hence the result.
\end{proof}

Following remarks are immediate from the proof of Theorem \ref{t3}.

\begin{remark}\label{aff-nzer-const}
If $f \in {\rm Aff}(B_n)$ such that $\vt \in {\rm supp}(f)$ then $f$ is a nonzero constant map.
\end{remark}

\begin{remark}\label{r.re-su-n}
Given $f \in {\rm Aff}(B_n)_n$, there exist $k, q \in [n]$ and $\sigma \in S_n$ such that ${\rm supp}(f) = \{(i, k) \mid i \in [n]\}$, ${\rm Im}(f) = \{(i\sigma, q) \mid i \in [n]\} \cup\{\vt\}$ and $(i, k)f = (i\sigma, q)$, for all $i \in [n]$.
\end{remark}

\begin{definition}\label{d.re-su-n}
For $f \in {\rm Aff}(B_n)_n$, a \emph{representation of $f$} is defined by a triplet $(k, q; \sigma)$, where the parameters $k, q$ and $\sigma$ are as per Remark \ref{r.re-su-n}.
\end{definition}

\subsection{Classification of elements in $A^+(B_n)$}\label{si-a+b}

We conclude the section in this subsection by obtaining the cardinality of  $A^+(B_n)$ along with a classification of its elements  (cf. Theorem \ref{t2}).

\begin{proposition}\label{p.vt-su}
If $f \in A^+(B_n)$ and $\vartheta \in {\rm supp}(f)$, then $f$ is a nonzero constant map.
\end{proposition}

\begin{proof}
For $f \in A^+(B_n)$, write $f = f_1 + \cdots + f_m$ where each $f_j \in {\rm Aff}(B_n)$. If $\vt f \ne \vt$, then each $f_j$ must be a nonzero constant map (cf. Remark \ref{aff-nzer-const}) and hence $f$ is a nonzero constant map.
\end{proof}

It is clear that any nonzero constant map in $M(B_n)$ is of full support. The following corollary of Proposition \ref{p.vt-su} ascertains that the converse holds in case of the elements in $A^+(B_n)$.

\begin{corollary}\label{c.fu-cst}
If $f \in A^+(B_n)$ is of full support, then $f$ is a nonzero constant map. Hence, $|A^+(B_n)_{n^2+1}| = n^2.$
\end{corollary}

\begin{proposition}\label{sum-su-n}
If $f, g \in {\rm Aff}(B_n)_n$ and $h$ is a nonzero constant map,  then we have
\begin{enumerate}
\item $|{\rm supp}(g + f)| = 0$ or $1$,
\item $|{\rm supp}(h + f)| = 1$,
\item $|{\rm supp}(f + h)| = 0$ or $n$.
\end{enumerate}
\end{proposition}

\begin{proof}
Let $(k, q; \sigma)$ and $(k', q'; \sigma')$ be the representations of $f$ and $g$, respectively and $h = \xi_{(r, s)}$, for $r, s \in [n]$.
\begin{enumerate}
\item If $k \ne k'$, then ${\rm supp}(f) \cap {\rm supp}(g) = \vn$ and clearly, $|{\rm supp}(g + f)| = 0$. Otherwise,
\[{\rm supp}(f) = {\rm supp}(g) = \{(i, k) \mid 1 \le i \le n\}.\]
Since $\sigma$ is a permutation on $[n]$, let $j$ be the unique element in $[n]$ such that $j\sigma = q'$. Now, \[(j, k)(g + f) = (j\sigma', q') + (j\sigma, q) = (j\sigma', q)\] and, for $i \in [n]$ with $i \ne j$, $(i, k)(g + f) = \vt$. Thus, ${\rm supp}(g + f) = \{(j, k)\}$.

\item Since $\sigma$ is a permutation on $[n]$, there is a unique $t \in [n]$ such that $t\sigma = s$. Now,
\[(t, k)(h + f) = (t, k)\xi_{(r, s)} + (t, k)f = (r, s) + (t\sigma, q) = (r, q)\] and for all $\alpha \in B_n \setminus \{(t, k)\}$, $\alpha(h + f) = \vt$. Thus, $|{\rm supp}(h + f)| = 1$.

\item If $q \ne r$, then clearly, $\alpha(f + h) = \vt$, for all $\alpha \in B_n$, so that  $|{\rm supp}(f + h)| = 0$. Otherwise, for all $1 \le i \le n$, \[(i, k)(f + h) = (i\sigma, q) + (r, s) = (i\sigma, s)\] and, for all $\alpha \in B_n \setminus {\rm supp}(f)$, $\alpha(f + h) = \vt$. Hence, $|{\rm supp}(f + h)| = n$.
\end{enumerate}
\end{proof}

\begin{lemma}\label{l12d}
For $f \in M(B_n)$ and $k, l, p, q \in [n]$,  if $(k, l)f = (p, q)$ and $\alpha f = \vt$ for all $\alpha \in B_n \setminus \{(k, l)\}$, then $f \in A^+(B_n)_{1}$. Hence, $|A^+(B_n)_{1}| = n^4.$
\end{lemma}

\begin{proof}
It is sufficient to prove that $f$ is a finite sum of affine maps. Consider a permutation $\sigma$ on $[n]$ such that $k\sigma = q$ and then consider $g \in {\rm Aff}(B_n)_n$ whose representation is $(l, q; \sigma)$. Note that $\xi_{(p, q)} + g \in A^+(B_n)$. Moreover,
\[(k, l)(\xi_{(p, q)} + g) =  (p, q) + (k\sigma, q) = (p, q)\] and, for $(i, l) \in {\rm supp}(g)$ with $i \ne k$, $(i, l)(\xi_{(p, q)} + g) = (p, q) + (i\sigma, q) = \vt$, as $i\sigma \ne q$. Further, it is clear that $\alpha(\xi_{(p, q)} + g) = \vt$, for all $\alpha \notin {\rm supp}(g)$. Hence, $\xi_{(p, q)} + g = f$. Consequently, $|A^+(B_n)_{1}|$ is the number of choices of $k, l, p, q \in [n]$, as desired.
\end{proof}

\begin{notation}
We use $\sis{(k, l)}{(p, q)}$ to denote the singleton support map $f$ whose ${\rm supp}(f) = \{(k, l)\}$ and ${\rm Im}(f) \setminus \{\vt\} = \{(p, q)\}$.
\end{notation}

\begin{lemma}\label{su-a+}
If $f \in  A^+(B_n)\setminus {\rm Aff}(B_n)$, then $|{\rm supp}(f)| = 1$.
\end{lemma}

\begin{proof}
Suppose that $f = f_1 + \cdots + f_m \in A^+(B_n) \setminus {\rm Aff}(B_n)$ with $m \ge 2$ and $f_i = g_i + h_i$, for some $g_i \in End(B_n)$ and $h_i \in \mathcal{C}_{B_n}$. In view of Theorem \ref{t3}, for each $i$, either $f_i \in  \mathcal{C}_{B_n}$ or $f_i \in {\rm Aff}(B_n)_n$. Clearly, none of the $f_i$'s can be $\xi_\vt$. For all $i \ge 2$, if $f_i$'s are nonzero constant maps, then $f = g_1 + \xi$, where $\xi = h_1 + f_2 + \cdots + f_m \in \mathcal{C}_{B_n}$ so that $f \in {\rm Aff}(B_n)$; this contradicts the choice of $f$. On the other hand, $f_j \in {\rm Aff}(B_n)_n$ for some $j \ge 2$. Note that $|{\rm supp}(f)| \le |{\rm supp}(h_{j-1} + f_j)|$. Hence,  by Proposition \ref{sum-su-n}(2), $|{\rm supp}(f)| \le 1$; consequently, $|{\rm supp}(f)| = 1$.
\end{proof}

In view of Theorem \ref{t3}, we have the following corollaries of Lemma \ref{su-a+}.

\begin{corollary}\label{l12e}
For $n \geq 3$ and $1 < k <n, \; A^+(B_n)_{k} = \varnothing$.
\end{corollary}

\begin{corollary}\label{l12c}
For $n \ge 1$, $f \in {\rm Aff}(B_n)_n \Longleftrightarrow  f \in A^+(B_n)_n$. Hence, $|A^+(B_n)_{n}| = (n!)n^2.$
\end{corollary}

Now, combining the results from Corollary \ref{c.fu-cst} through Corollary \ref{l12c}, we have the following main result of the section.

\begin{theorem}\label{t2}
For $n \geq 2$, $|A^+(B_n)| = (n!+1)n^2+n^4+1.$
In fact, we have the following breakup of the elements of $A^+(B_n)$.
\begin{enumerate}
\item The  number of mappings of full support  is $n^2.$
\item The number of mappings of $n$-support  is $(n!)n^2.$
\item The number of mappings of singleton support is $n^4.$
\item The  number of mappings of  $0$-support is $1$.
\end{enumerate}
\end{theorem}

\begin{remark}\label{n = 1}
For $n =1$, $End(B_n) = {\rm Aff}(B_n) = A^+(B_n) = \{(1, 1; id)\} \cup \mathcal{C}_{B_n}$, where $id$ is the identity permutation on $[n]$. Note that all the elements of the near-semiring $A^+(B_1)$ are idempotents in both the semigroup reducts.
\end{remark}

\section{Structure of $(A^+(B_n), +)$}

In this section, we study the additive semigroup structure of the affine near-semiring $A^+(B_n)$ via Green's relations. First we characterize Green's classes of $(A^+(B_n), +)$ and find their sizes.  Further, we investigate the regular and idempotent elements of $A^+(B_n)$ and certain relevant subsemigroups. We observe that the semigroup $(A^+(B_n), +)$ is eventually regular.  We refer \citep{hw76} for certain fundamental notions on semigroups.

The semigroup reduct $(A^+(B_n), +)$ of the affine near-semiring $(A^+(B_n), +, \circ)$ is denoted by $A^+(B_n)^{^+}$. Further, in a particular context, if there is no emphasis on the semigroup, we may simply write $A^+(B_n)$.

\subsection{Green's classes of $A^+(B_n)^{^+}$}

 In this subsection, we study all the Green's relations $\mathcal{R}$, $\mathcal{L}$, $\mathcal{D}$, $\mathcal{J}$ and $\mathcal{H}$ on the semigroup $A^+(B_n)^{^+}$. Being a finite semigroup, $A^+(B_n)^{^+}$ is periodic; hence, by \cite[Proposition 1.5]{hw76}, the Green's relations $\mathcal{J}$  and $\mathcal{D}$ coincide on $A^+(B_n)^{^+}$. For $f \in A^+(B_n)^{^+}$, $R_f$, $L_f$  and $D_f$ denote the Green's classes of the relations $\mathcal{R}$, $\mathcal{L}$ and $\mathcal{D}$, respectively, containing $f$.

The following result is useful in characterizing the Green's classes of $A^+(B_n)^{^+}$.

\begin{proposition}\label{p.sub-sg}
In $A^+(B_n)^{^+}$, we have the following.
\begin{enumerate}
\item The set of constant maps $\mathcal{C}_{B_n}$ is a subsemigroup which is isomorphic to $B_n$.
\item The set $A^+(B_n)_1 \cup \{\xi_\vt\}$ is an ideal which is isomorphic to the $0$-direct union of $n^2$ copies of $B_n$.
\end{enumerate}
Hence, both the subsemigroups are regular.
\end{proposition}

\begin{proof}\
\begin{enumerate}
\item Clearly, the assignment $\alpha \mapsto \xi_\alpha$, for all $\alpha \in B_n$, is an isomorphism from $B_n$ to $\mathcal{C}_{B_n}$.

\item Observe that, by Lemma \ref{r2}, $A^+(B_n)_1 \cup \{\xi_\vt\}$ is an ideal. Consider the semigroup $Z$ which is $0$-direct union of the collection $\{B_n^{(i, j)} \mid (i, j) \in [n] \times [n]\}$ of $n^2$ copies of $B_n$ indexed by $[n] \times [n]$. For the nonzero elements of $Z$, we write $(p, q)^{(i, j)}$ to denote the element $(p, q)$ which is in the $(i, j)$th copy of $B_n$.  Now, the assignment $\sis{(k, l)}{(p, q)} \mapsto (p, q)^{(k, l)}$ and $\xi_\vt \mapsto \vt$ is clearly a semigroup isomorphism from $A^+(B_n)_1 \cup \{\xi_\vt\}$ to $Z$.
\end{enumerate}
Regularity of these semigroups follows from the regularity of $B_n$.
\end{proof}

\begin{lemma}\label{l10}
Let $f$ and $g$ be two mappings in the semigroup $(M(B_n), +)$. If  $f\mathcal{R}g$ (or $f\mathcal{L}g$), then ${\rm supp}(f) = {\rm supp}(g)$.
\end{lemma}

\begin{proof}
If $f = g$, the result is straightforward. Otherwise, if $f\mathcal{R}g$, then there exist $h,h' \in M(B_n)$ such that $f+h = g$ and $g+h' = f$. By Lemma \ref{r2},\[ {\rm supp}(g) = {\rm supp}(f+h) \subseteq {\rm supp}(f)\] and \[{\rm supp}(f) = {\rm supp}(g+h') \subseteq {\rm supp}(g).\] Hence,  ${\rm supp}(f) = {\rm supp}(g)$. Similarly, if  $f\mathcal{L}g$, then ${\rm supp}(f) = {\rm supp}(g)$.
 \end{proof}

For $1 \le i \le 2$, let $\pi_i : [n] \times [n] \rightarrow [n]$ be the $i$th projection map. That is, $(p, q)\pi_1 = p$ and $(p, q)\pi_2 = q$, for all $(p, q) \in [n] \times [n]$.

\begin{definition}
For $f \in M(B_n)$, an \emph{image invariant} of $f$,  denoted by $ii(f)$, is defined as the number $q \in [n]$, if exists, such that \[{\rm Im}(f) \setminus \{\vt\} = \{(i,q) \mid i \in X\}\] for some $X \subseteq [n]$.
\end{definition}

\begin{remark}\label{im-inv}
From Theorem \ref{t2}, it can be observed that every nonzero element of  $A^+(B_n)$ has an image invariant.
\end{remark}

\begin{remark}\label{R-Bn}
For nonzero elements of $B_n$, the relation $\mathcal{R}$ (or $\mathcal{L}$) is the equality on the first coordinate (or on the second coordinate, respectively). Hence, other than the class $\{\vt\}$, the number of $\mathcal{R}$ or $\mathcal{L}$ classes in $B_n$ is $n$. Consequently, any two nonzero elements of $B_n$ are $\mathcal{D}$-related (cf. \cite[Lemma 2.4]{hw76}).
\end{remark}

In view of Lemma \ref{l10}, we characterize the Green's relations $\mathcal{R}$, $\mathcal{L}$ and $\mathcal{D}$ on $A^+(B_n)^{^+}$ classified by the supports of its elements.

\begin{theorem}\label{t.chr-singnfs}
For $f, g \in A^+(B_n)_1 \cup A^+(B_n)_{n^2 + 1}$, we have
\begin{enumerate}
\item $f \mathcal{R} g$ if and only if ${\rm supp}(f) = {\rm supp}(g)$ and $\alpha f\pi_1 = \alpha g\pi_1$, $\forall \alpha \in {\rm supp}(f)$,

\item  $f \mathcal{L} g$ if and only if  ${\rm supp}(f) = {\rm supp}(g)$  and  $\alpha f\pi_2 = \alpha g\pi_2$, $\forall \alpha \in {\rm supp}(f)$,

\item  $f \mathcal{D} g$ if and only if ${\rm supp}(f) = {\rm supp}(g).$
\end{enumerate}
\end{theorem}

\begin{proof}\
\begin{enumerate}
\item In view of Lemma \ref{l10}, if both $f, g$ are in $A^+(B_n)_{n^2 + 1}$ or in $A^+(B_n)_1$, the characterization follows from Proposition \ref{p.sub-sg} and Remark \ref{R-Bn}.
\item Similar to (1).
\item Clearly, $f \mathcal{D} g$ implies ${\rm supp}(f) = {\rm supp}(g)$. The converse follows from Proposition \ref{p.sub-sg} and Remark \ref{R-Bn}.
\end{enumerate}
\end{proof}

\begin{theorem}\label{t.chr-sun}
For $f,g \in A^+(B_n)_n$, we have
\begin{enumerate}
\item $f \mathcal{R} g$ if and only if ${\rm supp}(f) = {\rm supp}(g)$ and $\alpha f\pi_1 = \alpha g\pi_1$, $\forall \alpha \in {\rm supp}(f)$,
\item $L_f = \{f\}$,
\item $f \mathcal{D} g$ if and only if $f \mathcal{R} g$.
\end{enumerate}
\end{theorem}

\begin{proof}\
\begin{enumerate}
\item For $f, g \in A^+(B_n)_n$,  by Proposition \ref{sum-su-n}(3), $f \mathcal{R} g$ if and only if there exist $\xi, \xi' \in \mathcal{C}_{B_n}$ such that $f = g + \xi$ and $g = f + \xi'$. This implies that $\alpha f\pi_1 = \alpha g\pi_1$, for all $\alpha \in {\rm supp}(f)$ (= ${\rm supp}(g)$, by Lemma \ref{l10}). For the converse, let $ii(f) = l$ and $ii(g) = m$ (cf. Remark \ref{im-inv}). Choose the functions $h = \xi_{(l, m )}$ and $h' = \xi_{(m, l)}$. We show, simultaneously, that ${\rm supp}(f+h) = {\rm supp}(g)$ and $\alpha(f + h) = \alpha g$, for all $\alpha \in {\rm supp}(g)$. Note that ${\rm supp}(f + h) \subseteq {\rm supp}(f) = {\rm supp}(g)$. Let $\alpha \in {\rm supp}(g)$. Then, $\alpha f \neq \vt$ so that $\alpha f = (k, l)$ for some $k \in [n]$, as $ii(f) = l$. Thus, since $\alpha f\pi_1 = \alpha g\pi_1$ and $ii(g) = m$, we have $\alpha g = (k, m)$. Now,
\[\alpha (f + h) = \alpha f + \alpha h = (k,l) + (l,m) = (k,m) = \alpha g.\]
Hence, $\alpha \in {\rm supp}(f + h)$ and $f + h = g$.  Similarly, we can prove that $g + h' = f$ so that $f\mathcal{R}g$.

\item If $n = 1$, the result is straightforward. For $n \ge 2$, let $g \in L_f$ with $g \ne f$. Then there exists $h \in A^+(B_n)$ such that $h + f = g$. However, by  Proposition \ref{sum-su-n}, $|{\rm supp}(h + f)| \le 1$; a contradiction. Thus, $L_f = \{f\}$.

\item Follows from (2).
\end{enumerate}
\end{proof}

In view of  Theorem \ref{t2}, we have the following corollary of theorems \ref{t.chr-singnfs} and \ref{t.chr-sun}.
\begin{corollary}\label{a+bn-gcl-ct}
For $n \ge 2$, we have the following.
\begin{enumerate}
\item The number of $\mathcal{R}$-classes in $A^+(B_n)^{^+}$ is $(n!)n + n^3 + n + 1$.
\item The number of $\mathcal{L}$-classes in $A^+(B_n)^{^+}$ is $(n!)n^2 + n^3 + n + 1.$
\item The number of $\mathcal{D}$-classes in $A^+(B_n)^{^+}$ is  $(n!)n + n^2 + 2.$
\end{enumerate}
\end{corollary}

\begin{proof}\
\begin{enumerate}
\item Other than the class $\{\xi_\vt\}$, by Theorem \ref{t.chr-singnfs}(1), there are $n$ and $n^3$ $\mathcal{R}$-classes containing the full support maps and singleton support maps, respectively. Further, by Theorem \ref{t.chr-sun}(1), $(n!)n$ $\mathcal{R}$-classes are present in $A^+(B_n)_n$ (each is of size $n$). Hence, we have the total number.

\item Similar to above (1), by Theorem \ref{t.chr-singnfs}(2), there are $n^3 + n + 1$  $\mathcal{L}$-classes containing singleton support and constant maps. And the remaining number $(n!)n^2$ is the number of $\mathcal{L}$-classes containing $n$-support maps (cf. Theorem \ref{t.chr-sun}(2)).

\item Other than the class $\{\xi_\vt\}$, by Theorem \ref{t.chr-singnfs}(3),  all the full support elements form a single $\mathcal{D}$-class and there are $n^2$ $\mathcal{D}$-classes (each is of size $n^2$) containing singleton support elements. Including $(n!)n$ $\mathcal{D}$-classes which are present in $A^+(B_n)_n$ (cf. Theorem \ref{t.chr-sun}(3) and above (1)), we have $(n!)n + n^2 + 2$ $\mathcal{D}$-classes in $A^+(B_n)^{^+}$.
\end{enumerate}
\end{proof}

\begin{remark} \label{r.ape-mbn}
Since $\alpha  + \alpha = \alpha + \alpha + \alpha$, for all $\alpha \in B_n$, we have $f + f = f + f + f$, for all $f$ in the semigroup $(M(B_n), + )$. Consequently, any subsemigroup of $(M(B_n), +)$ is aperiodic.
\end{remark}

Hence, by Remark \ref{r.ape-mbn}, we have the following proposition (cf. \cite[Proposition 4.2]{pin86}).

\begin{proposition}\label{a+bn-jnh}
The Green's relation $\mathcal{H}$  is trivial on the semigroup $A^+(B_n)^{^+}$.
\end{proposition}

\begin{remark}
Proposition 3.3(e) in \citep{a.gilbert10b}, given for endomorphism near-semirings, also follows immediately from Remark \ref{r.ape-mbn}.
\end{remark}

\subsection{Regular elements and idempotents in $A^+(B_n)^{^+}$}

In this subsection, we characterize the regular and idempotent elements in $A^+(B_n)^{^+}$ and ascertain that $A^+(B_n)^{^+}$ is eventually regular. We observe that the set of regular elements in $A^+(B_n)^{^+}$ forms an inverse semigroup.

\begin{theorem}\label{l27} For $n \ge 2$,
\begin{enumerate}
\item $f \in  A^+(B_n)^{^+}$ is of $k$-support with $k \ne n$ if and only if $f$ is regular;
\item $I(A^+(B_n)^{^+}) = \displaystyle\left\{\xi_{\alpha} \left| \alpha \in I(B_n) \right.\right\} \cup \left\{\left.\sis{(i, j)}{(k, k)} \right| i, j, k \in [n]\right\}.$
\end{enumerate}
\end{theorem}

\begin{proof}\
\begin{enumerate}
\item In view of Proposition \ref{p.sub-sg}, it is sufficient to show that $n$-support elements are not regular. For $f \in A^+(B_n)_n$, if there is a $g \in A^+(B_n)$ such that $f + g + f = f$, then, by Proposition \ref{sum-su-n}, $|{\rm supp}(f + g + f)| \le 1$; a contradiction. Hence, $f$ is not regular.

\item Since $I(B_n) = \{(k, k) \mid k \in [n]\}$, by Proposition \ref{p.sub-sg}, \[I(\mathcal{C}_{B_n}) = \{\xi_\alpha \mid \alpha \in I(B_n)\}\] and \[I(A^+(B_n)_1) = \displaystyle \left\{\left.\sis{(i, j)}{(k, k)} \right| i, j, k \in [n]\right\}.\] Further, for $f \in A^+(B_n)_n$, since $|{\rm supp}(f+f)| = 1$ (cf. Proposition \ref{sum-su-n}),  $f + f$ cannot be $f$. Hence, the idempotents of $A^+(B_n)^{^+}$ are merely in $\mathcal{C}_{B_n} \cup A^+(B_n)_1$.
\end{enumerate}
\end{proof}

\begin{corollary}
$A^+(B_n)^{^+}$ has $n^3 + n + 1$ idempotents and $n^4 + n^2 + 1$ regular elements.
\end{corollary}

For $n \ge 2$, since $n$-support  elements in $A^+(B_n)^{^+}$ are not regular, the semigroup $A^+(B_n)^{^+}$ is not a regular semigroup. However, in the following proposition, we prove that $A^+(B_n)^{^+}$ is eventually regular, i.e. for every $f \in A^+(B_n)^{^+}$, we observe that there is a number $m$ such that $mf$ $(= f + \cdots + f$ for $m$ times) is regular \citep{ed83}.

\begin{proposition}\label{l14e}
The semigroup $A^+(B_n)^{^+}$ is eventually regular.
\end{proposition}

\begin{proof}
In view of Theorem \ref{l27}, it remains to show that, for each $f \in A^+(B_n)_n$, there is a number $m$ such that $mf$ is regular. Now, for $f \in A^+(B_n)_n$,  since $f + f \in A^+(B_n)_1$ (cf.  Proposition \ref{sum-su-n}),  we have $2f$ is regular. Hence, the semigroup $A^+(B_n)^{^+}$ is eventually regular.
\end{proof}

For $n \geq 2$, let $K$ be the set of regular elements in $A^+(B_n)^{^+}$. By Theorem \ref{l27}, $K = A^+(B_n) \setminus A^+(B_n)_n$. Further, by Theorem \ref{l27}(2), all the idempotents of $A^+(B_n)^{^+}$ are in $K$. We prove the following theorem concerning the set $K$.

\begin{theorem}\label{l27b}
 $(K, +)$ is an inverse semigroup.
\end{theorem}

\begin{proof}
Let $f, g \in K$. If one of them is the zero map, then $|{\rm supp}(f + g)| = 0$. If one of them is of singleton support, then by Lemma \ref{r2}, $|{\rm supp}(f + g)| \le 1$. If both $f$ and $g$ are of full support, then $|{\rm supp}(f + g)| = n^2+1$ or $0$. Thus, in any case, $f + g \in K$. Hence, $(K, +)$ is a regular semigroup.

Referring to \cite[Theorem 1.2]{hw76}, it is sufficient to show that the idempotents in $K$ commute. Let $f, g  \in I(K)$. If one of them is the zero map, then $f + g = g + f = \xi_\vt$. Otherwise, we have the following cases.
\begin{description}
  \item[{\rm\em Case 1}] $|{\rm supp}(f)| = |{\rm supp}(g)|$. Then,
   \[f+g = g+f =  \begin{cases}
                f & \text {if $f  = g$; }  \\
                \xi_\vt     & \text {otherwise.}
                  \end{cases} \]
  \item[{\rm\em Case 2}] $|{\rm supp}(f)| \neq |{\rm supp}(g)|$. Say, $|{\rm supp}(f)| = n^2+1$ and $|{\rm supp}(g)| = 1$. Then, \[f+g = g+f = \begin{cases}
                g & \text {if $\text{Im}(f) = \text{Im}(g) \setminus \{\vartheta\}$; }  \\
                \xi_\vt     & \text {otherwise.}
                  \end{cases} \]
\end{description}
Thus, $I(K)$ is commutative. Hence, $(K, +)$ is an inverse semigroup
\end{proof}

\section{Structure of $(A^+(B_n), \circ)$}

In this section, we study the multiplicative semigroup structure of the affine near-semiring  $A^+(B_n)$ via Green's relations. First we characterize Green's classes of $(A^+(B_n), \circ)$ and find their sizes.  We observe that the semigroup $(A^+(B_n), \circ)$ is  regular (cf. Theorem \ref{t.rec}) and orthodox (cf. Theorem \ref{t.oc}). Further, we investigate the idempotent elements of $A^+(B_n)$ and certain relevant subsemigroups.

The semigroup reduct $(A^+(B_n), \circ)$ of the affine near-semiring $(A^+(B_n), +, \circ)$ is denoted by $A^+(B_n)^{^\circ}$. Further, in a particular context, if there is no emphasis on the semigroup, we may simply write $A^+(B_n)$. For $f, g \in M(B_n)$, the product $f \circ g$  will simply be written as $fg$.

\subsection{Green's classes of $A^+(B_n)^{^\circ}$}

As mentioned earlier in the case of $A^+(B_n)^{^+}$, the Green's relations $\mathcal{J}$  and $\mathcal{D}$ coincide also on the semigroup $A^+(B_n)^{^\circ}$. In this subsection, we study the Green's relations $\mathcal{R}$, $\mathcal{L}$, $\mathcal{D}$ and $\mathcal{H}$ on  $A^+(B_n)^{^\circ}$.  For $f \in A^+(B_n)^{^\circ}$, the Green's classes containing $f$ of the relations $\mathcal{R}$, $\mathcal{L}$, $\mathcal{D}$ and $\mathcal{H}$ on $A^+(B_n)^{^\circ}$ are again denoted by $R_f$, $L_f$, $D_f$ and $H_f$, respectively.

We begin with the structural properties of the subsemigroups $A^+(B_n)_1 \cup \{\xi_\vt\}$ and $\mathcal{C}_{B_n}$  of  $A^+(B_n)^{^\circ}$.
\begin{remark}\label{a+bnc-fs-subsgp}
The semigroup $(\mathcal{C}_{B_n}, \circ)$ has right zero multiplication, i.e. $fg = g$, for all $f, g \in \mathcal{C}_{B_n}$. Consequently, $(\mathcal{C}_{B_n}, \circ)$ is regular.
\end{remark}

\begin{proposition}\label{a+bnc-sis-subsgp}
 The semigroup $(A^+(B_n)_1 \cup \{\xi_\vt \}, \circ)$ is isomorphic to $(B_{n^2}, +)$. Hence, $(A^+(B_n)_1 \cup \{\xi_\vt \}, \circ)$ is regular.
\end{proposition}

\begin{proof}
Clearly, the assignment $\sis{(k, l)}{(p, q)} \mapsto ((k, l), (p, q))$ and $\xi_\vt \mapsto \vt$ is a semigroup isomorphism from $(A^+(B_n)_1 \cup \{\xi_\vt \}, \circ)$ to $(B_{n^2}, +)$. Hence, since the semigroup $(B_{n^2}, +)$ is regular, $(A^+(B_n)_1 \cup \{\xi_\vt \}, \circ)$ is regular.
\end{proof}

\begin{lemma}\label{l18}
If $g$ is a nonconstant map in $A ^+(B_n)$, then ${\rm supp}(fg) \subseteq {\rm supp}(f)$.
\end{lemma}
\begin{proof}
If $fg$ is the zero map, then the result is straightforward. Let  $fg \ne \xi_{\vt}$ and $\alpha \in {\rm supp}(fg)$. Then, $\vt \ne \alpha (fg)  = (\alpha f)g$ so that  $\alpha f \in {\rm supp}(g)$. Since $g$ is not a constant map, by Proposition \ref{p.vt-su},  $\alpha f \neq \vartheta$ so that $\alpha \in {\rm supp}(f).$ Hence, ${\rm supp}(fg) \subseteq {\rm supp}(f)$.
\end{proof}

We present a characterization of the Green's relation $\mathcal{R}$ on $A^+(B_n)^{^\circ}$ in the following theorem.

\begin{theorem}\label{t.chr-rc}
For $f, g \in A^+(B_n)^{^\circ}$, we have
\begin{enumerate}
\item if $f, g \in \mathcal{C}_{B_n}$, then $f \mathcal{R} g$;
\item if $f, g \not\in \mathcal{C}_{B_n}$, then $f \mathcal{R} g   \Longleftrightarrow {\rm supp}(f) = {\rm supp}(g).$
\end{enumerate}
Moreover, for $n \ge 2$, the number of $\mathcal{R}$-classes in $A^+(B_n)^{^\circ}$ is $n^2 + n + 1.$
\end{theorem}

\begin{proof}$\;$
\begin{enumerate}
\item By Remark \ref{a+bnc-fs-subsgp}, any two elements of $\mathcal{C}_{B_n}$ are $\mathcal{R}$-related. Thus, $\mathcal{C}_{B_n}$ has a single $\mathcal{R}$-class containing all the constant maps.

\item If $f \mathcal{R} g$ with $f \ne g$, then there exist $h,h' \in A ^+(B_n)$ such that $fh = g$ and $gh' = f$. Note that $h$ and $h'$ are nonconstant maps; otherwise, $f$ and $g$ will be constant maps. Now,  by Lemma \ref{l18}, $ {\rm supp}(g) = {\rm supp}(fh) \subseteq {\rm supp}(f)$ and  ${\rm supp}(f)  = {\rm supp}(gh') \subseteq {\rm supp}(g)$. Hence, ${\rm supp}(f) = {\rm supp}(g)$.

Conversely, suppose ${\rm supp}(f) = {\rm supp}(g)$. If $f,g \in A^+(B_n)_1$, by Remark \ref{R-Bn} and Proposition \ref{a+bnc-sis-subsgp}, we get $f \mathcal{R} g$. Consequently, $A^+(B_n)_1$ has $n^2$ $\mathcal{R}$-classes each of size $n^2$. On the other hand, let $f = (k,p ; \sigma)$ and $g = (k,q; \sigma')$. Since $\sigma$ and $\sigma'$ are permutation on $[n]$, define the bijection $\tau : [n] \rightarrow [n]$ by $i \sigma \tau = i \sigma'$, for all $i \in[n]$. Now consider the $n$-support maps $h = (p, q; \tau)$ and $h' = (q, p; \tau^{-1})$ in $A^+(B_n)$.
For $i \in [n]$, we have $$(i, k)(fh) = (i\sigma, p)h = (i \sigma \tau, q) = (i \sigma', q) = (i, k)g;$$ and, for
$\alpha \in B_n \setminus {\rm supp}(f), \alpha (fh) = \vt h = \vt = \alpha g$. Thus, $fh = g$. Similarly,  $gh' = f$. Hence, $f \mathcal{R} g$. Thus, $A^+(B_n)_n$ contains $n$ $\mathcal{R}$-classes each of size $(n!)n$.
\end{enumerate}
Hence, for $n \ge 2$, the number of $\mathcal{R}$-classes in $A^+(B_n)^{^\circ}$ is $n^2 + n + 1.$
\end{proof}

We present a characterization of the Green's relation $\mathcal{L}$ on $A^+(B_n)^{^\circ}$ in the following theorem.

\begin{theorem}\label{t.chr-lc}
 For $f,g \in A ^+(B_n)^{^\circ}$, $f \mathcal{L}g \Longleftrightarrow {\rm Im}(f) = {\rm Im}(g)$. Moreover, for $n \ge 2$, the number of $\mathcal{L}$-classes in $A^+(B_n)^{^\circ}$ is $2n^2 + n + 1.$
\end{theorem}

\begin{proof}
If $f \mathcal{L} g$ with $f \ne g$, then there exist $h,h' \in A ^+(B_n)$ such that $hf= g$ and $h'g = f$. Since ${\rm Im}(g) = {\rm Im}(hf) \subseteq {\rm Im}(f)$ and $ {\rm Im}(f) = {\rm Im}(h'g) \subseteq {\rm Im}(g)$, we have ${\rm Im}(f) = {\rm Im}(g)$.

Conversely, suppose ${\rm Im}(f) = {\rm Im}(g)$ so that $|{\rm supp}(f)| = |{\rm supp}(g)|$. Clearly, the $\mathcal{L}$-classes in $\mathcal{C}_{B_n}$ are singletons. Consequently, $\mathcal{C}_{B_n}$ has $n^2 + 1$  $\mathcal{L}$-classes. If  $f,g \in A^+(B_n)_1$, by Remark \ref{R-Bn} and Proposition \ref{a+bnc-sis-subsgp},  we have $f \mathcal{L} g$ and hence, there are $n^2$  $\mathcal{L}$-classes in $A^+(B_n)_1$ each is of size $n^2$. Otherwise, let $f = (k,p; \sigma)$ and $g = (l,p; \rho)$. Set $h = (l,k; \rho\sigma^{-1})$ and $h' = (k,l; \sigma\rho^{-1}) \in A^+(B_n)$. For $i \in [n]$, we have \[(i, l)(hf) = (i \rho\sigma^{-1}, k)f = (i \rho\sigma^{-1} \sigma, p) = (i \rho, p) = (i, l)g;\] and, for $\alpha \in B_n \setminus {\rm supp}(h)$, \[ \alpha(hf) = \vt f = \vt = \alpha g,\] as ${\rm supp}(h) = {\rm supp}(g)$. Thus, $hf = g$. Similarly, we can observe that  $h'g = f$. Hence, $f \mathcal{L} g$. Consequently, we have $n$ $\mathcal{R}$-classes containing $n$-support elements each is of size $(n!)n$.

Hence,  for $n \ge 2$, the number of $\mathcal{L}$-classes in $A^+(B_n)^{^\circ}$ is $2n^2 + n + 1.$
\end{proof}

Following characterization of the Green's relation $\mathcal{H}$ on $A^+(B_n)^{^\circ}$ is a corollary of theorems \ref{t.chr-rc} and \ref{t.chr-lc}.

\begin{corollary}\label{t.chr-hc}
For $f, g \in A^+(B_n)^{^\circ}$, $f \mathcal{H}g $ if and only if  ${\rm Im}(f) = {\rm Im}(g)$ and ${\rm supp}(f) = {\rm supp}(g)$.
Moreover, for $n \geq 2$, the number of $\mathcal{H}$-classes in $A^+(B_n)^{^\circ}$ is $n^4 + 2n^2 + 1.$
\end{corollary}

We characterize the  Green's relation $\mathcal{D}$ in the following theorem.

\begin{theorem}\label{acd}
For $f,g \in A^+(B_n)^{^\circ}$, $f \mathcal{D}g \Longleftrightarrow |{\rm supp}(f)| = |{\rm supp}(g)|$ or $f, g \in \mathcal{C}_{B_n}$. Hence, for $n \ge 2$, the number of $\mathcal{D}$-classes in $A^+(B_n)^{^\circ}$ is $3$.
\end{theorem}

\begin{proof}
 For $f,g \in A^+(B_n)^{^\circ}$, observe that
\begin{eqnarray*}
f \mathcal{D} g &\Rightarrow& \mbox{there exists } h \in A^+(B_n)^{^\circ} \mbox{ such that } f \mathcal{L} h \mbox{ and } h \mathcal{R} g\\
&\Rightarrow& {\rm Im}(f) = {\rm Im}(h) \mbox{ and } h \mathcal{R} g\; (\mbox{by Theorem \ref{t.chr-lc}})\\
&\Rightarrow& |{\rm supp}(f)| = |{\rm supp}(h)| \mbox{ and } (\mbox{either } {\rm supp}(h) = {\rm supp}(g) \mbox{ or } h, g \in \mathcal{C}_{B_n})\\
&&(\mbox{by Theorem \ref{t.chr-rc}})\\
&\Rightarrow& \mbox{either } |{\rm supp}(f)| = |{\rm supp}(g)| \mbox{ or } f, g \in \mathcal{C}_{B_n}\\
\end{eqnarray*}
Conversely, if $f, g \in \mathcal{C}_{B_n}$, then by Theorem \ref{t.chr-rc}, $f \mathcal{R} g$ so that $f \mathcal{D} g$. If $f, g \in A^+(B_n)_1$, then, by Remark \ref{R-Bn} and Proposition \ref{a+bnc-sis-subsgp}, $f \mathcal{D} g$. Finally, let $f, g \in A^+(B_n)_n$ such that $f = (k, p; \sigma)$ and  $g = (l, q; \rho)$. For $\tau \in S_n$, consider $h = (l,p; \tau) \in A^+(B_n)$. Now, by Theorem \ref{t.chr-lc} and Theorem \ref{t.chr-rc}(2), $ f \mathcal{L} h$ and $ h \mathcal{R} g$ so that $ f \mathcal{D} g.$

Hence, for $n \ge 2$, $A^+(B_n)^{^\circ}$ has three $\mathcal{D}$-classes, viz. $\mathcal{C}_{B_n}$, $A^+(B_n)_1$ and $A^+(B_n)_n$.
\end{proof}

\subsection{Regular elements and idempotents in $A^+(B_n)^{^\circ}$}

In this subsection, we characterize the regular and idempotent elements in $A^+(B_n)^{^\circ}$ and ascertain that $A^+(B_n)^{^\circ}$ is regular. Moreover, it is an orthodox semigroup. We observe that the set excluding the full support elements in $A^+(B_n)^{^\circ}$ forms an inverse semigroup.

\begin{theorem}\label{t.rec}
The semigroup $A^+(B_n)^{^\circ}$ is regular.
\end{theorem}

\begin{proof}
In view of Remark \ref{a+bnc-fs-subsgp} and Proposition \ref{a+bnc-sis-subsgp}, it is sufficient to show that $n$-support elements in $A^+(B_n)^{^\circ}$ are regular.  Let $f  = (k, p; \sigma) \in A^+(B_n)_n$.  Set $g  = (p, k; \sigma^{-1}) \in A^+(B_n)$. For $(i, k) \in {\rm supp}(f)$ with $i \in [n]$, \[(i, k)(fgf) = (i \sigma, p)(gf) = (i \sigma \sigma^{-1}, k)f = (i, k)f.\]  Hence, $fgf = f$ so that $f$ is regular. Thus, the semigroup $A^+(B_n)^{^\circ}$ is regular.
\end{proof}

Now, in the following theorem, we identify the idempotent elements in $A^+(B_n)^{^\circ}$ and count their number.

\begin{theorem}\label{t.ic}
For $n \ge 2$, \[I(A^+(B_n)^{^\circ}) = \{\xi_\alpha \mid \alpha \in B_n\} \cup \{(k, k; id) \mid k \in [n]\} \cup \left\{\left.\sis{(i, j)}{(i, j)} \right| i, j \in [n]\right\}.\] Hence, $|I(A^+(B_n)^{^\circ})| = 2n^2 + n + 1$.
\end{theorem}

\begin{proof}
By Remark \ref{a+bnc-fs-subsgp}, clearly, all the $n^2 + 1$ elements of  $(\mathcal{C}_{B_n}, \circ)$ are idempotents. Since nonzero idempotents in $B_{n^2}$  are of the form $((i, j), (i, j))$, we have, $n^2$ elements of the form $\sis{(i, j)}{(i, j)}$ are idempotents in $A^+(B_n)_1$. Observe that the idempotent elements in $A^+(B_n)_n$ are of the form $(k, k; id)$, where $k \in [n]$ and $id$ is the identity permutation on $[n]$. Thus, there are  $n$ idempotent elements in $A^+(B_n)_n$. Hence, for $n \ge 2$, the number of idempotents in $A^+(B_n)^{^\circ}$ is $2n^2 + n + 1$.
\end{proof}

\begin{theorem}\label{t.oc}
The semigroup $A^+(B_n)^{^\circ}$ is orthodox.
\end{theorem}

\begin{proof}
In view of Theorem \ref{t.rec}, it is sufficient to prove that  $I(A^+(B_n)^{^\circ})$ is  a subsemigroup of $A^+(B_n)^{^\circ}$. Let $f, g \in I(A^+(B_n)^{^\circ})$. Note that if $f$ or $g$ is a constant map, then $fg$ is also a constant map and hence, $fg$ is an idempotent element.
Otherwise, we consider the following cases to show that $fg \in I(A^+(B_n)^{^\circ})$.
\begin{description}
\item[{\rm\em Case 1}] $f, g \in A^+(B_n)_l$, for $l \in \{1, n\}$. It can be observed that if $f  = g$, then $fg = f$; otherwise, $fg = \xi_\vt$.
\item[{\rm\em Case 2}] $f =$ $\sis{(i, j)}{(i, j)}$ and $g = (k, k; id)$. Observe that  if $j = k$, then $fg = gf = f$; otherwise, $fg = gf = \xi_\vt$.
\end{description}
Thus, the set  $ I(A^+(B_n)^{^\circ})$ is closed with respect to composition. Hence, $A^+(B_n)^{^\circ} $ is an orthodox semigroup.
\end{proof}

For $n \geq 2$, let $N = A^+(B_n) \setminus A^+(B_n)_{n^2 + 1}$. If $f,g  \in N$, then by Proposition \ref{p.vt-su}, $\vt \notin {\rm supp}(f) \cap {\rm supp}(g)$. Hence,  $\vt(fg) = \vt$ so that $|{\rm supp}(fg)| \neq n^2+1$. Thus, $N$ is closed with respect to composition. From the proof of Theorem \ref{t.rec}, it can be observed that $(N, \circ)$ is regular. Also, from the proof of Theorem \ref{t.oc}, the set $I(N)$ is closed with respect to composition. Further, note that $(I(N), \circ)$ is a commutative semigroup. Hence, we have the following theorem.

\begin{theorem}\label{l28}
$(N, \circ)$ is an inverse semigroup.
\end{theorem}

\section{An example: $A^+(B_2)$}

In this section, we illustrate our results using the affine near-semiring $A^+(B_2)$. First note that $A^+(B_2)$ has 29 elements (cf. Theorem \ref{t2}). The number of $\mathcal{D}$-classes in $A^+(B_2)^{^+}$ is 10 and it is 3 in $A^+(B_2)^{^\circ}$ (cf. Corollary \ref{a+bn-gcl-ct} and Theorem \ref{acd}). The number of $\mathcal{L}$-classes in $A^+(B_2)^{^+}$ is 19 and it is 11 in $A^+(B_2)^{^\circ}$ (cf. Corollary \ref{a+bn-gcl-ct} and Theorem \ref{t.chr-lc}). The number of $\mathcal{R}$-classes in $A^+(B_2)^{^+}$ is 15 and it is 7 in $A^+(B_2)^{^\circ}$ (cf. Corollary \ref{a+bn-gcl-ct} and Theorem \ref{t.chr-rc}). Since the $\mathcal{H}$-relation is trivial on $A^+(B_2)^{^+}$, all the 29 elements are in 29 different classes. The number of $\mathcal{H}$-classes in $A^+(B_2)^{^\circ}$ is 25 (cf. Corollary \ref{t.chr-hc}). All this information along with the respective Green's classes of both the semigroups $A^+(B_2)^{^+}$ and $A^+(B_2)^{^\circ}$ are shown in \textsc{Figure} \ref{f.eb} using egg-box diagrams. Here, following the notations/representaions introduced in this paper, the elements of $A^+(B_2)$ are displayed with their supports and images. Thus, the characterizations of the respective Green's relations can also be crosschecked in this figure. Further, in the figure, the idempotents elements in these semigroups are marked with a * on their left-top corner.

\begin{figure}[htp]
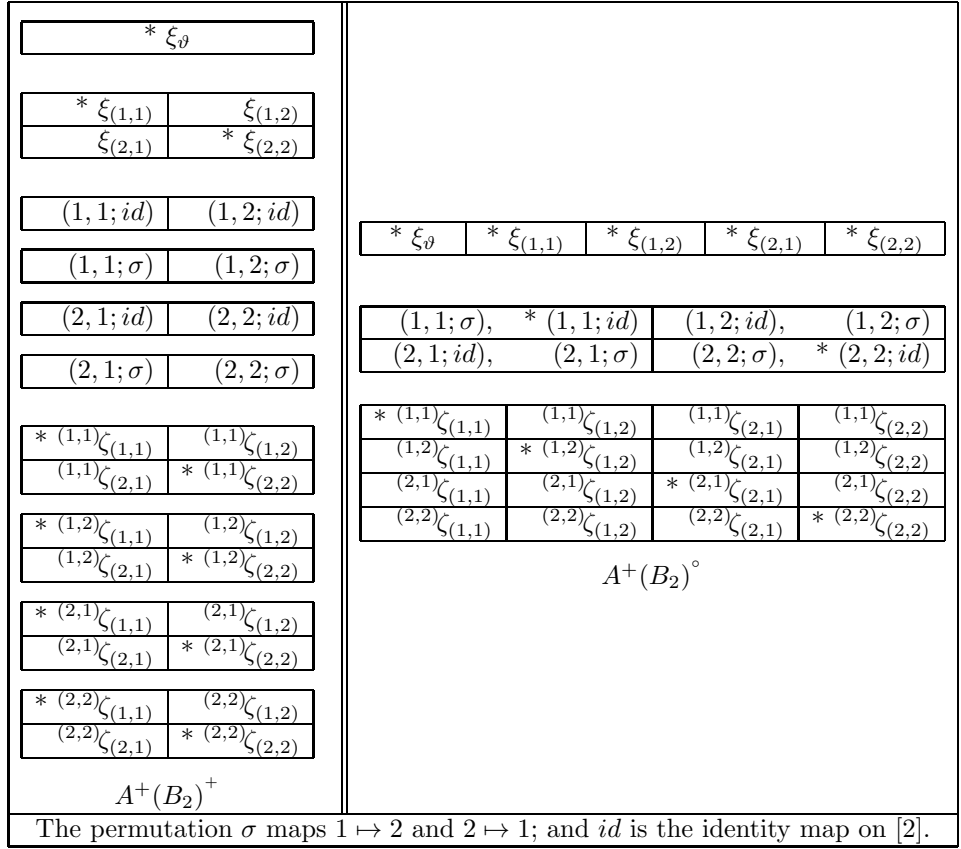

\begin{center}
\begin{tabular}{|c||c|}
\hline
\begin{tabular}{|r|r|}
\noalign{\vskip 0.25cm}
\hline  \multicolumn{2}{|c|}{*  $ \xi_\vt$}   \\
\hline
\noalign{\vskip 0.5cm}
\hline  * $\xi_{(1,1)}$ & $\xi_{(1,2)}$   \\
\hline    $\xi_{(2,1)}$ &  * $\xi_{(2,2)}$   \\
\hline
\noalign{\vskip 0.5cm}
\hline    $(1, 1; id)$ & $(1, 2; id)$   \\ \hline
\noalign{\vskip 0.25cm}
\hline    $(1, 1; \sigma)$ & $(1, 2; \sigma)$   \\ \hline
\noalign{\vskip 0.25cm}
\hline    $(2, 1; id)$ & $(2, 2; id)$   \\ \hline
\noalign{\vskip 0.25cm}
\hline    $(2, 1; \sigma)$ & $(2, 2; \sigma)$   \\ \hline
\noalign{\vskip 0.5cm}
\hline   * $\sis{(1, 1)}{(1, 1)}$ & $\sis{(1, 1)}{(1, 2)}$   \\
\hline    $\sis{(1, 1)}{(2, 1)}$ & * $\sis{(1, 1)}{(2, 2)}$   \\
\hline
\noalign{\vskip 0.25cm}
\hline   * $\sis{(1, 2)}{(1, 1)}$ & $\sis{(1, 2)}{(1, 2)}$   \\
\hline    $\sis{(1, 2)}{(2, 1)}$ & * $\sis{(1, 2)}{(2, 2)}$   \\
\hline
\noalign{\vskip 0.25cm}
\hline   * $\sis{(2, 1)}{(1, 1)}$ & $\sis{(2, 1)}{(1, 2)}$   \\
\hline    $\sis{(2, 1)}{(2, 1)}$ & * $\sis{(2, 1)}{(2, 2)}$   \\
\hline
\noalign{\vskip 0.25cm}
\hline  *  $\sis{(2, 2)}{(1, 1)}$ & $\sis{(2, 2)}{(1, 2)}$   \\
\hline    $\sis{(2, 2)}{(2, 1)}$ &  * $\sis{(2, 2)}{(2, 2)}$   \\
\hline
\noalign{\vskip 0.25cm}
\multicolumn{2}{c}{$A^+(B_2)^{^+}$}\\
\end{tabular}\hspace{.2cm} &
\begin{tabular}{rrrrrrrrrrrrrrrrrrrr}
\hline
\multicolumn{4}{|c|}{* $\xi_\vt$}&\multicolumn{4}{|c|}{* $\xi_{(1,1)}$}&\multicolumn{4}{|c|}{* $\xi_{(1,2)}$}&
\multicolumn{4}{|c|}{* $\xi_{(2,1)}$}&\multicolumn{4}{|c|}{* $\xi_{(2,2)}$}\\
\hline
\noalign{\vskip 0.25cm}
&&&&&&&&&&&&&&&&&&&\\
\hline
\multicolumn{5}{|r}{$(1, 1; \sigma)$,} & \multicolumn{5}{r|}{* $(1, 1; id)$} & \multicolumn{5}{|r}{$(1, 2; id)$,} & \multicolumn{5}{r|}{$(1, 2; \sigma)$}\\
\hline
\multicolumn{5}{|r}{$(2, 1; id)$,} & \multicolumn{5}{r|}{$(2, 1; \sigma)$} & \multicolumn{5}{|r}{$(2, 2; \sigma)$,} & \multicolumn{5}{r|}{ * $(2, 2; id)$}\\
\hline
&&&&&&&&&&&&&&&&&&&\\
\hline
\multicolumn{5}{|r|}{* $\sis{(1, 1)}{(1, 1)}$} & \multicolumn{5}{|r|}{$\sis{(1, 1)}{(1, 2)}$} &\multicolumn{5}{|r|}{$\sis{(1, 1)}{(2, 1)}$} &\multicolumn{5}{|r|}{$\sis{(1, 1)}{(2, 2)}$} \\
\hline
\multicolumn{5}{|r|}{ $\sis{(1, 2)}{(1, 1)}$} & \multicolumn{5}{|r|}{* $\sis{(1, 2)}{(1, 2)}$} &\multicolumn{5}{|r|}{$\sis{(1, 2)}{(2, 1)}$} &\multicolumn{5}{|r|}{$\sis{(1, 2)}{(2, 2)}$} \\
\hline
\multicolumn{5}{|r|}{$\sis{(2, 1)}{(1, 1)}$} & \multicolumn{5}{|r|}{$\sis{(2, 1)}{(1, 2)}$} &\multicolumn{5}{|r|}{* $\sis{(2, 1)}{(2, 1)}$} &\multicolumn{5}{|r|}{$\sis{(2, 1)}{(2, 2)}$} \\
\hline
\multicolumn{5}{|r|}{$\sis{(2, 2)}{(1, 1)}$} & \multicolumn{5}{|r|}{$\sis{(2, 2)}{(1, 2)}$} &\multicolumn{5}{|r|}{$\sis{(2, 2)}{(2, 1)}$} &\multicolumn{5}{|r|}{* $\sis{(2, 2)}{(2, 2)}$} \\
\hline
\noalign{\vskip 0.25cm}
\multicolumn{20}{c}{$A^+(B_2)^{^\circ}$}\\
\end{tabular}\\
\hline
\multicolumn{2}{|c|}{The permutation $\sigma$ maps $1 \mapsto 2$ and $2 \mapsto 1$; and $id$ is the identity map on $[2]$.}\\
\hline
\end{tabular}
\caption{Egg-box Diagrams for $A^+(B_2)^{^+}$ (left) and $A^+(B_2)^{^\circ}$ (right)}\label{f.eb}
\end{center}
\end{figure}

\section*{Acknowledgements}

We would like to express our sincere gratitude to the referee for his/her insightful comments on the paper which have changed the shape of the paper drastically.


\end{document}